\documentclass[10pt]{amsart} 
\title{Euler characteristics of generalized Haken manifolds}

\author{Michael W. Davis}
\address{Department of Mathematics,
The Ohio State University,
231 W. 18th Ave.,
Columbus, Ohio 43210}
\email{davis.12@math.osu.edu}
\urladdr{}

\author{Allan L. Edmonds}
\address{Department of Mathematics, Indiana University, 831 E. 3rd St., Bloomington, IN 47401}
\email{edmonds@indiana.edu}
 \theoremstyle{plain} 
 \newtheorem{theorem}{Theorem}[section]
 \newtheorem{conjecture}[theorem]{Conjecture}
  \newtheorem{signconj}[theorem]{Euler Characteristic Sign Conjecture}
  
\newtheorem{proposition}[theorem]{Proposition}
\newtheorem{corollary}[theorem]{Corollary}
\newtheorem{lemma}[theorem]{Lemma}
 
\theoremstyle{remark}
\newtheorem{remark}[theorem]{Remark}

\newtheorem{examples}[theorem]{Examples}

 \newcommand{\wt}{\widetilde}
\newcommand{\wh}{\widehat}

\newcommand{\cu}{\mathcal {U}} 
\newcommand{\rr}{{\mathbb R}} 
\newcommand{\zz}{{\mathbb Z}} 
\newcommand{\gs}{\sigma} 

\newcommand{\gl}{\lambda}
\newcommand{\gG}{\Gamma} 
\newcommand{\chorb}{\chi^{\mathrm{orb}}}

\newcommand{\lk}{\operatorname{Lk}}
\newcommand{\GHS}{\operatorname{GHS}}
\newcommand{\cone}{\operatorname{Cone}}
\newcommand{\codim}{\operatorname{codim}}
\newcommand{\cat}{{\mathrm{CAT}}}
\newcommand{\comment}[1]{}

\begin{document}

\begin{abstract}    
Haken $n$--manifolds have been defined and studied by B. Foozwell and H. Rubinstein in analogy with the classical Haken manifolds of dimension 3, based upon the the theory of boundary patterns developed by K. Johannson. The Euler characteristic of a Haken manifold is analyzed and shown to be equal to the sum of the Charney-Davis invariants of the duals of the boundary complexes of the $n$--cells at the end of a hierarchy. These dual complexes are shown to be flag complexes. It follows that the Charney-Davis Conjecture is equivalent to the Euler Characteristic Sign Conjecture for Haken manifolds. Since the Charney-Davis invariant of a flag simplicial 3--sphere is known to be non-negative  it follows that a closed Haken $4$--manifold has non-negative Euler characteristic. These results hold as well for generalized Haken manifolds whose hierarchies can end with compact contractible manifolds rather than cells.
\end{abstract}
\maketitle


\section{Introduction} \label{s:intro}
Haken $n$--manifolds, for $n>3$,  were defined and studied by B. Foozwell and H. Rubinstein \cite{Foozwell2007, Foozwell2011, FoozwellRubinstein2011,FoozwellRubinstein2012} in analogy with the classical Haken manifolds of dimension 3, building on the notion of a boundary pattern, developed in dimension 3 by K. Johannson \cite{Johannson1979,Johannson1994}. Foozwell \cite{Foozwell2007, FoozwellRubinstein2011} proved that they are aspherical and indeed have universal covering space homeomorphic to euclidean space \cite{Foozwell2007, Foozwell2011}. 

These manifolds can be 
 endowed with a hierarchy, that is, a prescription for successively cutting open the manifold until a disjoint union of $n$--cells,  with a simple regular cell structure on the boundary induced by the cutting submanifolds. In general these Haken cells do not induce a  cell complex structure on the original manifold. Nonetheless, we make use of the hierarchy to compute the Euler characteristic of the Haken manifold in terms of the cell structure of the Haken cells at the end of the hierarchy. It turns out that the Euler characteristic is equal to the sum of the Charney-Davis invariants of the simplicial spheres dual to the simple cell structures on the Haken cells.

 A key conceptual observation is that  manifolds with boundary patterns may be viewed as right-angled orbifolds with an orbifold Euler characteristic that is invariant under the process of cutting open along a hypersurface.

 We also explain how to generalize the notion of Haken manifolds in such a way as to allow arbitrary compact contractible manifolds at the end of a hierarchy, not just cells. Such manifolds are still aspherical but allow the possibility that the universal covering need not be Euclidean space. 

We  show that the simplicial spheres dual to the boundary complexes of the associated Haken cells are flag simplicial complexes. Thus the classical Euler Characteristic Conjecture about even-dimensional closed aspherical manifolds is reduced for closed generalized Haken manifolds to the Charney-Davis Conjecture for flag generalized simplicial spheres. In particular the Euler Characteristic Conjecture holds for all closed generalized Haken 4--manifolds. An earlier and more computational proof of the latter result (in the case of ordinary Haken 4--manifolds) appears in Edmonds \cite{Edmonds2013}.

Full statements and definitions will be given subsequent sections.

In Section \ref{s:patterns} we analyze the orbifold Euler characteristic that we associate with a manifold with boundary pattern and show that it is invariant under cutting open along a %
 hypersurface. In Section \ref{s:cells} we give a combinatorial interpretation of the notion of a Haken (homotopy) cell, concluding with examples of Haken manifolds arising from $\cat(0)$ cubical manifolds. 
 In the final section (Section \ref{s:euler}) we apply the earlier results to the Euler Characteristic Sign Conjecture for even-dimensional aspherical manifolds.

\subsubsection*{Acknowledgement}The research of first author was partially supported by  NSF grant DMS 1007068.

\section{Boundary Patterns and Orbifolds }\label{s:patterns}
We begin with the most basic aspects of Haken $n$--manifolds, concentrating on manifolds with boundary patterns, deferring the full definitions of Haken cells and Haken manifolds to a subsequent section.

\subsection{Boundary patterns} A \emph{boundary pattern} for an $n$--manifold is a decomposition of the boundary into connected $(n-1)$--manifolds such that the intersection of any $k$ of them is either empty or an $(n-k)$-submanifold. The elements of the boundary pattern are called \emph{facets}. A a component of a nonempty  intersection of facets is a \emph{stratum}. The relative interior of a stratum is a \emph{pure stratum}. 
The facets are codimension-one strata. By convention each component of the manifold itself is a codimension-0 stratum.

The boundary pattern is \emph{complete} if the union of its facets is the entire boundary. All boundary patterns considered here will be complete.  Notice that each facet inherits an induced boundary pattern. We refer to the entire configuration of facets and their intersections as the \emph{boundary complex}.

The \emph{nerve} of the boundary complex is the abstract simplicial complex $L$ with a vertex
for each facet and a $(k-1)$--simplex for each nonempty $k$-fold intersection.  
(The empty simplex corresponds to the whole manifold, i.e., to the codimension-$0$ stratum.)  For each simplex $\gs$ of $L$, let $S_\gs$ denote the corresponding union of strata. 

\subsection{Simple cells and homotopy cells}\label{ss:cells}
A \emph{simple $n$--cell} is a compact $n$--manifold with boundary pattern such that each stratum  is homeomorphic to $D^{n-k}$ where $k$ is its codimension.  If  each stratum is only required to be a compact contractible manifold, then we have a (simple) \emph{homotopy $n$--cell}.  If $c$ is a simple $n$--cell, then the nerve $L_c$ of its boundary complex is called its \emph{dual}.  It is a triangulation of $S^{n-1}$.  Moreover, since the simpicial complex dual to the boundary complex of $S_\gs$ is $\lk(\gs)$ (the link of $\gs$ in $L$),  we have that $\lk(\gs)$ is homeomorphic to $S^{n-\dim\gs -1}$.  Similarly, if $M$ is a homotopy $n$--cell, then $L$ is a $(n-1)$-dimensional ``generalized homology sphere'' (abbreviated as $\GHS^{n-1}$).  (Recall that a simplicial complex $L$ is a \emph{generalized homology $(n-1)$-sphere}  if it is a polyhedral homology manifold with the same homology as $S^{n-1}$.)  

\begin{remark}
If a simplicial complex $K$ is a polyhedral homology $n$--manifold, then the link of each $p$--simplex $\gs$ is a $\GHS^{n-p-1}$.  One does not gain much by requiring $K$ to be a manifold -- there is no difference for $n\le 3$ and for $n\ge 4$ the only further requirement is that the link of each vertex be simply connected.

For a polyhedral homology $n$--manifold $K$ and $p$--simplex $\sigma\in K$, the \emph{dual cone} $D(\gs)$ to $\sigma$ is a certain subcomplex of the  barycentric subdivision $K$ which is isomorphic to the cone on the barycentric subdivision of $\lk(\gs)$.   So, $D(\gs)$ is a contractible polyhedral homology manifold with boundary.  If $\lk(\gs)$ is homeomorphic to $S^{n-p-1}$, then the  dual cone of $\gs$ is a simple $(n-p)$--cell.
\end{remark}

\begin{proposition}\label{p:dual}
Suppose a simplicial complex $L$ is a triangulation of $S^{n-1}$ and that for each simplex $\gs\in L$, $\lk(\gs)$ is homeomorphic to $S^{\codim \gs -1}$.  Then the space $\cone (L)$ naturally has  the structure of a simple $n$--cell.  \end{proposition}
\begin{proof}[Sketch of proof]
If $L'$ denotes the barycentric subdivision of $L$, then for each simplex $\gs \in L$ there is a subcomplex $D(\gs)$ of $L'$ called the ``dual cone'' of $\gs$ which is homeomorphic to a face of a simple cell structure on $\cone (L')$.  (In particular, each facet is the closed star in $L'$ of a vertex of $L$.)
\end{proof}
\begin{proposition}
Suppose a simplicial complex $L$ is a $\GHS^{n-1}$.  Then there is a simple homotopy $n$--cell $c$ such that the nerve of its boundary complex is $L$.  Moreover, $c$ is unique up to a strata-preserving homeomorphism.
\end{proposition}
\begin{proof}[Sketch of  proof]
Using different terminology the proof of this is explained in Theorem 2.2 of \cite{Davis2013}.  The main ingredient in the proof is the fact that every homology $m$-sphere bounds a contractible $(m+1)$--manifold \cite{Kervaire1969} (for $m\ne 3$).  (For $m=3$ this uses \cite{FreedmanQuinn1990}.)  Using this fact one constructs a ``resolution'' of $\cone (L)$ as in Sullivan \cite{Sullivan1971} (also compare \cite{Cohen1970}).  The result is a homotopy $n$--cell c, together with a cell-like map $c\to \cone (L')$ which takes each face of $c$ to the dual cone of its corresponding simplex.  The last sentence (uniqueness) follows from the $3$-dimensional Poincar\'e Conjecture and the fact that the topological $h$-cobordism theorem is true in evey dimension.
\end{proof}

\subsection{Right-angled orbifolds}\label{ss:rar}
An orbifold is  \emph{right-angled} if it is locally modeled on the action of $(\zz/2)^n$ on $\rr^n$ by reflections across the coordinate hyperplanes.
A manifold $M$ with boundary pattern naturally has the structure of a right-angled orbifold $O(M)$.  Each pure facet has local group $\zz/2$ and each pure stratum of codimension $k$ has local group $(\zz/2)^k$.  

Given $M$ a manifold with boundary pattern, we can calculate the orbifold Euler characteristic $\chorb(O(M))$ of the associated orbifold by assigning a weight of $(1/2)^k$ to each pure stratum of codimension $k$:
	\begin{equation}\label{e:chorb}
	\chorb(O(M))=\sum (1/2)^{\codim S} \chi(S,\partial S)
	\end{equation} 
where the sum is over all strata $S$ and where, as usual, the relative Euler characteristic is given by $\chi(S, \partial S)=\chi(S)-\chi(\partial S)$.  By Poincar\'e duality, $\chi(S,\partial S) = (-1)^{\dim S}\chi(S)$, so \eqref{e:chorb} can be rewritten as
	\begin{equation}\label{e:chorb2}
	\chorb(O(M))=(-1)^n\sum (-1/2)^{\codim S} \chi(S)).
	\end{equation}	
For example if $M$ is an $n$--cube, $n\ge 1$, with its natural boundary pattern, then $\chorb(O(M))=0$.
	
When $n$ is odd (and the boundary pattern is complete), $\chorb(O(M))=0$.  
\begin{proposition}[Manifold doubles]\label{p:doubles}
Suppose $O$ is a right-angled orbifold with $l$ facets.  Then there is a closed manifold $\widehat{O}$ and an action of $(\zz/2)^l$ on $\widehat{O}$ with quotient orbifold $O$.  
\end{proposition}
\begin{corollary}\label{c:chiorbformula}
If $O$ is a right-angled orbifold and $\widehat{O}$ is its manifold double, then 
\[
\chorb(O)=\frac{\chi(\widehat{O})}{2^l} .
\]
\end{corollary}
\begin{remark}\label{r:chiorbformula}
Given a manifold with boundary pattern $M$,  let $\widehat{M}$ ($=\widehat{O}(M)$) denote the manifold double of $O(M)$.  One could then take the formula in Corollary \ref{c:chiorbformula} as the definition of $\chorb(O(M))$. More generally if $(\zz/2)^m$ acts by reflections on a manifold $\wh{M}$ with orbifold quotient $O$, then
\[
\chorb(O)=\frac{\chi(\wh{M})}{2^m} .
\]
\end{remark}

\begin{examples}
If $M$ has empty boundary, then $\widehat{M}=M$. If $M$ has nonempty connected boundary, consisting of a single facet, then $\widehat{M}$ is the ordinary manifold double consisting of two copies of $M$ glued together along their boundary by the identity map. The manifold double of a closed interval is a circle formed out of four closed intervals suitably identified. If $M$ is  a 2--simplex 
 (a triangle), then $\widehat{M}$ is a 2--sphere tessellated by 8 right-angled spherical triangles.
\end{examples} 
\begin{proof}[Proof of Proposition \ref{p:doubles}]
This is essentially a special case of the ``basic construction'' of Chapter 5 in \cite{Davis2008}, which we outline in the present context. The facets of the Haken $n$--manifold $M$ give it a mirror structure. We label the facets $F_{s}$, $s\in S$, where $S$ is a set of cardinality $l$, viewed as the standard set of generators of the elementary abelian $2$-group $G$. For each $x\in M$ set 
\[
S(x) = \{s\in S: x\in F_{s} \}.
\]
For each nonempty subset $T\subset S$, let $G_{T}$ denote the subgroup of $G$ generated by the involutions in $T$.

Define an equivalence relation $\sim$ on $G\times M$ by setting
\[
(g,x)\sim (h,y) \text{ if and only if } x=y \text{ and } gh^{-1}\in G_{S(x)}
\]
and then set
\[
\widehat{M}  = (G\times M)/\sim.
\]
The manifold double $\widehat{M}$ in this case is denoted by $\cu(G,M)$ in \cite{Davis2008}, where this object is studied in much greater generality.
That $\widehat{M}$ is connected when $M$ is connected follows from Proposition 5.2.4 in \cite{Davis2008}.
That $\widehat{M}$ is an $n$--manifold follows from Proposition 10.1.10 in \cite{Davis2008}. The action of $G$ on $G\times M$, with orbit space $M$ clearly descends to an action of $G$ on $\widehat{M}$, with orbit space $M$.
\end{proof}

\subsection{Cutting open along a hypersurface}
We consider properly embedded, codimension-one,  submanifolds  (``hypersurfaces'') that meet the strata transversely. Such manifolds (or, more generally, maps) are called \emph{admissible}.  If we cut open along such an admissible hypersurface, the new manifold receives a boundary pattern in which the normal $S^0$-bundle over the hypersurface becomes a codimension one stratum. (If the hypersurface is two-sided, the $S^0$-bundle is trivial and each component of the hypersurface contributes two facets.) The remaining facets are obtained by cutting open the original facets along the boundary of the hypersurface.  If $M$ is the manifold with boundary pattern and $F$ is the hypersurface, then denote by $M\odot F$ the result of cutting $M$ open along $F$.

\begin{lemma}\label{l:main}
Suppose that $M'=M\odot F$ is the result of cutting $M$ open along a hypersurface.  Then
\[
\chorb(O(M'))=\chorb(O(M)).
\]
\end{lemma}

\begin{proof}
In the special case where $M$ is closed and $F$ is a closed submanifold, we let  $F'$ denotes the corresponding $S^0$-bundle over $F$. By \eqref{e:chorb},
\[
\chorb(O(M'))=\chi(M',F')+(1/2)\chi(F')=\chi(M)-\chi(F)+ 2(1/2)\chi(F)= \chi(M),
\]
The general case reduces to this special case by taking a $2^l$-fold cover using Corollary~\ref{c:chiorbformula}, where $l$ is the number of facets of $M$. Let $\widehat{M}$ denote  the manifold double of $O(M)$. Let $\widehat{F}$ be the pre-image of $F$ in $\widehat{M}$. Then $(\zz/2)^{l}$ acts on the manifold $\widehat{M}\odot \widehat{F}$ with orbifold quotient $O(M')$. 
  Thus 
  \begin{align*}
\chorb(O(M'))&= (1/2^{l})\chorb(O(\widehat{M}\odot \widehat{F})) \quad \text{by Remark \ref{r:chiorbformula}}\\
&=(1/2^{l})\chi(\widehat{M}) \quad \text{by the special case}\\
&=(1/2^{l})(2^{l})\chorb(O(M)) \quad \text{by Corollary \ref{c:chiorbformula}}\\
&=\chorb(O(M)).
\end{align*}
\end{proof}

\subsection{Prehierarchies}
A \emph{prehierarchy} for a compact $n$--manifold $M$ with a complete  
boundary pattern is a sequence of $n$--manifolds $M_{k}$ with complete 
boundary patterns and hypersurfaces $F_{k}$:
$$(M_0,F_0), (M_{1}, F_{1}),\dots ,(M_m,F_m)$$
where $M_{0}=M$, $M_{k+1}=M_{k}\odot F_{k}$, and $M_{m+1}$, the result of cutting $M_m$ open along $F_m$, is a disjoint union of simple homotopy $n$--cells. 

\begin{theorem}\label{t:prehier}
Suppose $(M_0,F_0), \dots , (M_m,F_m)$ is a prehierarchy for $M=M_0$.  Then \[
\chorb(O(M))=\sum_{c} \chorb(O(c))
\]
where the sum is over the homotopy $n$--cells $c$ in $M_{m+1}$.
\end{theorem}

\begin{proof}
By Lemma~\ref{l:main}, $\chorb(O(M))= \chorb(O(M_{m+1}))$ and  $\chorb(O(M_{m+1}))$ is additive under disjoint union.
\end{proof}

\section{Haken Cells and Haken Manifolds}\label{s:cells}
Our goal here is to give a combinatorial characterization of Haken (homotopy) $n$--cells as having dual nerve a simplicial flag complex of dimension $n-1$. This requires delving somewhat more deeply into some of the intricacies of Haken cells.

\subsection{Useful boundary patterns}
We need to discuss the somewhat technical notion of a \emph{useful}  boundary pattern. A boundary pattern is said to be \emph{useful} if 
\begin{enumerate}
\item
whenever there is a loop in a single facet that is nullhomotopic in the manifold then it is nullhomotopic in the facet; 
\item
whenever there is a null-homotopic loop in the boundary consisting precisely of two arcs, each in distinct facets, then the loop bounds a 2-disk in the boundary meeting the intersection of the two facets in a single arc; and 
\item
whenever there is a null-homotopic loop in the boundary consisting precisely of three arcs, each in distinct facets, then the loop bounds a 2-disk in the boundary meeting the boundaries of the three facets in a single triod. 
\end{enumerate}
The slogan here is that ``small 2--disks are standard''.

Here we mainly need this notion in the case of a simply connected manifold. In this case a boundary pattern is \emph{useful} if and only if
\begin{enumerate}
\item
Each facet is simply connected;
\item
The intersection of any two facets is connected; and
\item
If three facets have pairwise nonempty intersections, then all three have nontrivial intersection.
\end{enumerate}

\subsection{Essential submanifolds}
Let $M$ be an $n$--manifold with boundary pattern.
We consider properly embedded, codimension-one submanifolds  (``hypersurfaces'') $F\subset M$ that meet the facets and their faces transversely.\footnote{Foozwell and Rubinstein only consider two-sided hypersurfaces. We allow hypersurfaces to be one-sided.}
If we cut open $M$ along such a submanifold $F$, the new manifold $M'$ inherits a natural boundary pattern in which the (one or) two components of the boundary of a tubular neighborhood of of the hypersurface become facets. The remaining facets are obtained by cutting the original facets open along the boundary of the hypersurface. Note that $F$ inherits a boundary pattern as well.

In order to ensure that we are describing an aspherical manifold, the hypersurfaces along which we cut are required to be \emph{essential}. The detailed properties required for the hypersurface  to be \emph{essential} will not concern us much here, but these properties include being injective on fundamental group, and a standard relative version of that condition.
In particular, any loop in $F$ that bounds a disk in $M$ also bounds a disk in $F$. 
In general this property ensures that the induced boundary pattern on $M\odot F$ is useful.
	See \cite{Edmonds2013,Foozwell2011,FoozwellRubinstein2011,FoozwellRubinstein2012} for more complete discussion.

\subsection{Haken cells and  Haken homotopy cells}
A  \emph{Haken homotopy $n$--cell} is defined inductively to be a topological homotopy $n$--cell with a complete useful  boundary pattern in which the facets are themselves  Haken homotopy cells. The definition in Foozwell-Rubinstein \cite{FoozwellRubinstein2011} of a Haken $n$--cell is the same except the word homotopy is omitted.  The inductive definition starts with $0$--cells, which are automatically Haken. Any closed interval with the unique complete boundary pattern is Haken. In dimension 2, a $p$-sided polygon is a Haken $2$--cell if and only if $p\ge 4$. It follows that a 2-dimensional face of a general Haken $n$--cell is a $p$-gon, with $p\ge 4$.

\subsection{Hierarchies} 
If $M$ is a manifold with  useful boundary pattern and $F\subset M$ is an essential codimension-one submanifold, then  we say that $(M,F)$ is a \emph{good pair}. 

A \emph{hierarchy} for a compact $n$--manifold $M$ with a complete useful boundary pattern is a prehierarchy
$$(M_0,F_0), (M_{1}, F_{1}),\dots ,(M_m,F_m)$$
consisting of good pairs, where each $M_k$ has a complete and useful boundary pattern and where $M_{m+1}$ is a disjoint union of Haken homotopy  $n$--cells.
By a \emph{generalized Haken $n$--manifold} we mean a compact $n$--manifold with a complete useful boundary pattern,  which admits a  hierarchy.\footnote{Foozwell and Rubinstein include a given hierarchy as part of the structure of a Haken manifold. They also require the essential codimension-one submanifolds $F_{k}$  to be two-sided. In addition these authors require that the end of the hierarchy consist of Haken $n$--cells, such that they and their faces are homeomorphic to topological cells.}

\begin{proposition}
A generalized Haken $n$-manifold is aspherical.
\end{proposition}
\begin{proof}The proof is modeled on that of Foozwell and Rubinstein \cite{FoozwellRubinstein2011}, Theorem 3.1, with modifications to allow for one-sided hypersurfaces and for generalized Haken cells.

The proof proceeds by induction on  the dimension of the manifold and the number of steps in a hierarchy. The cases when $n=1$ or 2 follow from the classification of manifolds in these dimensions. In addition, in any dimension a Haken manifold with a hierarchy of length 1 is just a collection of contractible manifolds, hence also aspherical.

Inductively, suppose that  Haken manifolds of smaller dimension or shorter length hierarchy than those of $M$ are aspherical. We may assume that $M$ and its cutting hypersurfaces are connected. If $M$ is cut open along the first hypersurface $F$ then the result is a manifold with boundary pattern $M'$ which is a Haken manifold with shorter hierarchy. By the induction hypothesis $M'$ is aspherical. If $F$ is two-sided, then the hierarchy for $M$ induces one on $F$, so induction on dimension also shows that $F$ is aspherical. If $F$ is one-sided, then the same argument shows that a suitable connected 2-fold covering $\widetilde{F}$ of $F$, given by the boundary of a tubular neighborhood $N$ of $F$, is aspherical. It follows from covering space theory that $F$ itself is aspherical in this case as well. 

The Seifert-van Kampen theorem shows that $\pi_{1}(M)$ is  a free product with amalgamation over $\pi_{1}(F)$ in the case when $F$ is two-sided and separating, $\pi_{1}(M)$ is  an HNN extension over $\pi_{1}(\widetilde{F})$ in the case when $F$ is one-sided, and $\pi_{1}(M)$ is an HNN extension over $\pi_{1}(F)$ in the case when $F$ is two-sided but nonseparating .

Now $\pi_{1}(F)\to \pi_{1}(M)$ is injective in the two-sided case, as in \cite{FoozwellRubinstein2011}, Theorem 3.1. The same argument shows that  in the one-sided case we have $\pi_{1}(\widetilde{F})\to \pi_{1}(M)$ injective.

Thus we see that in in the nonseparating cases $M$ can be described as the union of two compact aspherical manifolds $N$ and $\overline{M-N}$ intersecting along an aspherical manifold which is $\pi_{1}$-injective  into both $N$ and $\overline{M-N}$. In the separating case we may similarly write $M=M_{1}\cup_{F} M_{2}$. A classical theorem of J.~H.~C.~Whitehead then implies that $M$ is aspherical.
\end{proof}

\subsection{Characterization of Haken homotopy  cells}
As we saw in Subsection \ref{ss:cells}, 
the boundary complex of a simple $n$--cell may be viewed as the dual complex of a simplicial $(n-1)$-sphere and that of a simple homotopy $n$--cell as the (resolved) dual complex of  a $\GHS^{n-1}$. We now look more closely at the consequences of the Haken condition.

\begin{proposition}
If $X$ is a Haken homotopy $n$--cell, then for each $k$-face $S_\sigma$ of $X$, $k\le n$, the 1-skeleton of its dual simplicial $(k-1)$-sphere contains no empty triangle. (The dual simplicial $(k-1)$-sphere to $S_\sigma$ is identified with the link 
of the $(n-k-1)$--simplex $\sigma$ corresponding to $S_\sigma$ in the simplicial dual). \end{proposition}
\begin{proof}
In the case $\sigma=X$ the assertion is clear from the definition if $n\le 2$.   In general it is an interpretation of being a ``useful'' boundary pattern. Since all faces of a Haken homotopy cell are themselves Haken homotopy cells, the general result follows. 
\end{proof}

Recall that a simplicial complex in which any collection of $k+1$ pairwise adjacent vertices spans a $k$--simplex is called a \emph{flag simplicial complex}. Suggestively we think of a \emph{non-flag complex} as having a minimal \emph{empty simplex} of some dimension greater than 1, i.e., a subcomplex equivalent to the boundary of a $k$--simplex that does not actually span a $k$--simplex. 
\begin{lemma}\label{l:flaglink}
If $L$ is a flag simplicial complex and $\sigma\in L$, then $\lk(\sigma,L)$ is flag.
\end{lemma}
\begin{proof}
Let $\eta\subset \lk(\sigma,L)$ be a minimal empty simplex. Since $L$ is flag there is a simplex $\tau\in L$ such that $\eta=\partial \tau$. We need to show that $\tau\in \lk(\sigma,L)$. Now $\sigma*\partial \tau \cup \tau=\partial \rho$ for some $\rho\in L$ since $L$ is flag. But then $\rho=\sigma*\tau$, implying that $\tau\in \lk(\sigma,L)$.
\end{proof}
\begin{lemma}\label{l:link}
A simplicial complex $L$ is flag if and only if for each simplex $\sigma$ in $L$ (including the empty simplex) its link $\lk(\sigma,L)$ contains no empty triangle.
\end{lemma}
\begin{proof}
If $L$ is flag and $\sigma\in L$, then $\lk(\sigma,L)$ is flag by the preceding result, and hence contains no empty triangle.

For the converse assume that neither $L$ nor any link $\lk(\sigma,L)$ contains an empty triangle. We must show that $L$ is flag. To this end we proceed by induction (on dimension, say). Let $v_{0},\dots, v_{n}$ be vertices spanning a minimal non--simplex. By hypothesis we may assume that $n\ge 3$. Consider $\lk(v_{0},L)$. Note that $v_{1},\dots, v_{n}\in \lk(v_{0},L)$. Also all the edges $v_{i}v_{j}$ ($1\le i,j\le n$) lie in $\lk(v_{0},L)$, since  $v_{0}v_{i}v_{j}$ is a triangle of $L$ by the minimality hypothesis. By hypothesis $\lk(v_{0},L)$ contains no empty triangles, hence by induction $\lk(v_{0},L)$ is flag. Thus $v_{1}\cdots v_{n}$ is an $n$--simplex of $\lk(v_{0},L)$. But then $v_{0}v_{1}\cdots v_{n}$ is a simplex of $L$, as required.
\end{proof}

\begin{theorem}\label{t:flag}
A simple homotopy $n$--cell  is a  Haken homotopy $n$--cell if and only if the dual simplicial $\GHS^{n-1}$ is flag.
\end{theorem}

\begin{proof}
First suppose $M$ is a Haken homotopy $n$--cell. We need to argue that the  simplicial $(n-1)$-sphere $L$ dual to the boundary complex of $M$ is flag. It is part of the definition of a Haken homotopy $n$--cell that the simplicial dual of the boundary complex contains no empty triangle in its 1-skeleton. Since all faces of a Haken homotopy cell are themselves Haken homotopy cells there are no empty triangles in $\lk(\sigma, L)$ for any simplex $\sigma$ in $L$. By Lemma \ref{l:link} this implies $L$ is flag, as required.

Second suppose $M$ is a simple, homotopy $n$--cell with a simple regular homotopy cell complex structure on its boundary $(n-1)$--sphere, for which the dual simplicial generalized sphere $L$ is flag. We may assume that $n\ge 3$.  

The facets of $M$ are simple cells whose boundaries are duals of links of vertices, hence also flag by Lemma \ref{l:flaglink}. Therefore by induction the facets are Haken homotopy cells.

By definition the dual simplicial generalized sphere contains no empty triangles. We need to check that the given simple regular homotopy cell complex structure contains no triangles in its 1-skeleton. Such a 3-cycle would correspond to a triple $\sigma_{0}, \sigma_{1}, \sigma_{2}$ of $(n-1)$--simplices such that each intersection $\sigma_{i}\cap\sigma_{j}$ is an $(n-2)$--simplex ($i\ne j)$. Each $(n-1)$--simplex has $n$ vertices; each such pair of $(n-1)$-simplices has $n-1$ vertices in common. It follows that all three $(n-1)$--simplices have $n-2$ vertices in common. Then the set of all vertices involved in the three $(n-1)$-simplices has the property that each two vertices span an edge. The flag condition implies that these $n+1$ vertices span an $n$--simplex, contradicting the fact that we had a flag triangulation of a generalized $(n-1)$--sphere.
\end{proof}

\subsection{Some examples}\label{ss:exlocally amples} 
We describe a wide class of locally $\cat(0)$ manifolds in all dimensions that are Haken. Related discussion appears in \cite[Section 5]{FoozwellRubinstein2011}.
In contrast we point out examples of Haken manifolds that do not support locally $\cat(0)$ metrics. Finally, we indicate some standard examples of closed aspherical manifolds in higher dimensions that are not generalized Haken, even virtually.

\subsubsection{Locally $\cat(0)$-manifolds that are Haken}
We outline a general process of for imposing a Haken or generalized Haken structure on a closed manifold $M$ with a locally $\cat(0)$ cubical structure. The process always succeeds when the cubical structure on  $M$ arises from the action of a right-angled Coxeter group  $W$  associated with the dual Haken homotopy $n$--cell $X$ (also called a ``mirrored space'') corresponding to a  flag triangulation  $L$ of a $\GHS^{n-1}$. 
As in \cite{Davis2008} there is  a  cubical $\cat(0)$ structure on a  manifold $\mathcal{U}(W,X)$ with a free, cocompact action of $W$. Choosing a normal, torsion-free, finite index subgroup $\Gamma < W$, one obtains a closed aspherical, locally $\cat(0)$ manifold $M = \mathcal{U}(W,X)/\Gamma$. Such a manifold $M$ can be seen to be Haken, as we now explain in somewhat greater generality.

Suppose $M$ is a closed $n$-manifold with a locally $\cat(0)$, cubical structure.  Since $M$ is a polyhedral homology manifold, the link of each vertex is a $\GHS^{n-1}$, and since $M$ is an actual manifold  the link of each vertex is simply connected (assuming $n\ge 3$).  The universal cover $\wt{M}$ is a $\cat(0)$ cube complex.  The coordinate hyperplanes in each cube, extend to ``hyperplanes'' in the universal cover $\wt{M}$.  
The hyperplanes, and the intersections of hyperplanes, inherit a $\cat(0)$ cubical structure from $\wt{M}$. In general, these hyperplanes need only be homology submanifolds of codimension one; however, if the link of each cubical face
 is a simplicial sphere, then any hyperplane (as well as any intersection of hyperplanes) is an actual locally flat submanifold.  The image of a hyperplane in $M$ need not be an embedded homology submanifold (a ``hypersurface''); however, in many cases hypersurfaces are embedded.  For example, if the cubical structure comes from a cocompact action of a right-angled Coxeter group $W$ on $\wt{M}$ and if $\gG=\pi_1(M)$ is a normal, torsion-free subgroup of finite index in $W$, then the hypersurfaces are embedded by a lemma of Millson and Jaffee (cf.~\cite[Lemma 14.1.8]{Davis2008}). 
When the hypersurfaces are embedded, they can be used to define a hierarchy for $M$ (in a generalized sense where the hypersurfaces are only required to be homology submanifolds).   The ``cells'' at the end of the hierarchy are stars of vertices in the barycentric subdivision of the cubical complex, i.e., they are dual cones.  (When the links of vertices are simplicial spheres, these dual cones are actual simple cells.)  In the general case, one can replace each dual cone by its resolution by a homotopy cell (cf. \cite{Cohen1970, Sullivan1971}). The result is a manifold, which is homeomorphic to $M$, together with a collection of embedded hypersurfaces which are actual submanifolds.  The end of the hierarchy is  the collection of homotopy cells obtained by resolving the dual cones. 

\subsubsection{Manifolds that are Haken but not locally $\cat(0)$}
Many examples of Haken manifolds are not related to any locally $\cat(0)$ cubical structure.  If $\pi:M\to B$ is the projection map of a fiber bundle with fiber $\Sigma$ and if the base and fiber are both  Haken manifolds or generalized Haken manifolds, then so is $M$.  One easily constructs a hierarchy for $M$ from hierarchies for $B$ and $\Sigma$. 
To see this, note that if $F$ is  a hypersurface in $B$,   then $\pi^{-1}(F)$ is a hypersurface in $M$.  Hence, the inverse image of a hierarchy for $B$ yields the beginning of a hierarchy for $M$ which cuts $M$ into a disjoint union of manifolds of the form $\Sigma\times c$ where $c$ is a homotopy cell from the end of the hierarchy for $B$.  A hierarchy for $\Sigma$ then gives a hierarchy for $M$.  

If the bundle is not trivial, then even when the base and fiber have  locally $\cat(0)$ cubical structures one  cannot expect   $M$ to have such a structure. For example, if $M$ is an oriented $S^1$-bundle over $B$ and its Euler class in $H^2(B;\zz)$ does not have finite order, then $M$ does not admit a locally $\cat(0)$-metric.  (cf.~\cite[Theorem II.6.12]{BridsonHaefliger1999} and \cite[Lemma 12.1]{FrigerioLafontSisto2011}.)  

Another class of such examples arises from solvmanifolds.  Since any solvmanifold can be constructed via an iterated sequence of torus bundles, starting from a torus,  
solvmanifolds are Haken. However, if the fundamental group of the solvmanifold is not virtually  free abelian, then it does not admit a locally $\cat(0)$ metric (cf. the Solvable Subgroup Theorem \cite[Theorem II.7.8]{BridsonHaefliger1999}).

\subsubsection{Non-Haken aspherical manifolds}
Examples include irreducible, locally symmetric spaces of rank greater than 1.  On the one hand the fundamental group of an irreducible, locally symmetric space of rank greater than 1 has Kazhdan's Property T. For a general recent reference see the book of Bekka, de la Harpe, and Valette \cite{BekkadelaHarpeValette2008}.
On the other hand, the fundamental group of a Haken manifold splits as a nontrivial free product with amalgamation or as a nontrivial HNN extension. Such a splitting leads to an action of the group without a fixed point on a tree, which implies that the group does not have  Property T.  

\section{The Euler Characteristic Conjecture}\label{s:euler}
We apply the preceding work to the following fundamental conjecture about aspherical manifolds for a class of manifolds known as \emph{Haken $n$--manifolds}.
\begin{signconj}
If $M$ is a closed, aspherical manifold of even dimension $n = 2m$, then the Euler characteristic of $M$ satisfies $(-1)^m \chi(M) \geq 0.$
\end{signconj}
\noindent

The conjectured sign corresponds to the sign of the Euler characteristic of a product of $m$ surfaces of genus $g \ge 1$. This conjecture was first proposed as a question by W. Thurston in the 1970s. (See the Kirby problem set  \cite{Kirby1997}.) The first interesting and, in general, still unresolved case is in dimension 4.

Recall that we gave a formula \eqref{e:chorb} for the orbifold Euler characteristic as follows:
	\begin{equation*}
	\chorb(O(M))=\sum (1/2)^{\codim S} \chi(S,\partial S).
	\end{equation*} 
	
In terms of the nerve $L$, the orbifold Euler characteristic of a Haken $n$--manifold can be rewritten as:
	\[
	\chorb(O(M))=(-1)^n \sum_{\gs} (-1/2)^{\dim \gs +1} \chi (S_\gs)
	\]
where the sum is over all simplices in $L$, including the empty simplex.  If each stratum has Euler characteristic equal to 1 (e.g. if $M$ is a Haken homotopy $n$--cell) and $n$ is even, then this formula reads $\chorb(O(M))=\gl(L)$, where $\gl(L)$ is the \emph{Charney-Davis quantity} defined by
	\begin{equation*}\label{e:gl}
	\gl(L):= 1+\sum_{\gs\in L} (-1/2)^{\dim \gs +1}.
	\end{equation*}
So, Theorem~\ref{t:prehier} can be restated as follows.
\begin{theorem}\label{t:euler}
If $M$ is a closed generalized Haken $n$--manifold, $n=2m$, then 
\begin{equation*}
\chi(M) = \sum_{c} \lambda(L_{c})
\end{equation*}
where $c$ ranges over the Haken (homotopy) $n$--cells at the end of a hierarchy for $M$ and $L_{c}$ denotes the simplicial nerve associated with $c$.
\end{theorem}	

By Theorem \ref{t:flag} the dual nerve of the boundary complex of a Haken $n$--cell or Haken homotopy $n$--cell is a flag complex.

The Charney-Davis Conjecture may be stated as follows.
\begin{conjecture} [Charney and Davis \cite{CharneyDavis1995a}]
Let $L$ be a flag triangulated $(2k-1)$-dimensional sphere (or generalized homology sphere). Then 
\(
(-1)^{k}\lambda(L) \ge 0.
\)
\end{conjecture}
An immediate corollary of Theorems~\ref{t:prehier} and \ref{t:flag} is the following.
\begin{corollary}\label{c:CDEuler} 
The Charney-Davis Conjecture for generalized homology $(2k-1)$-spheres implies the Euler Characteristic Sign Conjecture for closed generalized Haken manifolds of  dimension $2k$. 
\end{corollary}

The Charney-Davis Conjecture is only known in the trivial case $k=1$ and the case $k=2$.
\begin{theorem} [Davis and  Okun \cite{DavisOkun2001}]
Let $L$ be a flag triangulated $3$--sphere (or homology $3$--sphere). Then 
\(
\lambda(L) \ge 0.
\)
\end{theorem}

Thus we have proved the following result.
\begin{corollary}
If $M$ is a closed generalized Haken $4$--manifold, then 
$\chi(M)   \ge 0$.
\end{corollary}

\bibliographystyle{amsalpha}
\bibliography{Davis-Edmonds}

\end{document}